\documentclass[12pt]{amsart}

\usepackage{amsmath, amssymb, amscd, amsthm, enumerate, a4wide}
\usepackage{hyperref}
\usepackage{color}

\newtheorem{theorem}{Theorem}[section]
\newtheorem{cor}[theorem]{Corollary}

\newtheorem{lemma}[theorem]{Lemma}
\newtheorem{prop}[theorem]{Proposition}
\newtheorem{Def}[theorem]{Definition}

\theoremstyle{remark}
\newtheorem{rem}[theorem]{Remark}
\numberwithin{equation}{section}

\newcommand{\R}{\mathbb R}
\newcommand{\N}{\mathbb N}
\newcommand{\C}{\mathbb C}
\newcommand{\Z}{\mathbb Z}
\newcommand{\T}{\mathbb T}

\newcommand{\al}{\alpha}

\newcommand{\ga}{\gamma}
\newcommand{\Ga}{\Gamma}
\newcommand{\de}{\delta}

\newcommand{\eps}{\varepsilon}

\newcommand{\te}{\theta}

\newcommand{\la}{\lambda}

\newcommand{\rphis}[5]{\,_{#1}\varphi_{#2}\!\left( \genfrac{.}{.}{0pt}{}{#3}{#4}
\,;#5 \right)}
\newcommand{\hf}{\frac{1}{2}}
\renewcommand{\emptyset}{\varnothing}

\newcommand{\mvert}{\mkern 2mu | \mkern 2mu}

\DeclareMathOperator{\sgn}{\mathrm{sgn}}
\newcommand{\Res}[1]{\underset{#1}{\mathrm{Res}}\,}

\begin{document}
\title{A solution to the Al-Salam--Chihara moment problem.}

\author{Wolter Groenevelt}
\address{Delft University of Technology\\
Delft Institute of Applied Mathematics\\
P.O.Box 5031\\
2600 GA Delft\\
The Netherlands}
\email{w.g.m.groenevelt@tudelft.nl}

\dedicatory{This paper is dedicated to Ben de Pagter on the occasion of his 65th birthday.}

\begin{abstract}
We study the $q$-hypergeometric difference operator $L$ on a particular Hilbert space. In this setting $L$ can be considered as an extension of the Jacobi operator for $q^{-1}$-Al-Salam--Chihara polynomials. Spectral analysis leads to unitarity and an explicit inverse of a $q$-analog of the Jacobi function transform. As a consequence a solution of the Al-Salam--Chihara indeterminate moment problem is obtained.
\end{abstract}


\maketitle

\section{Introduction}
For every moment problem there is corresponding set of orthogonal polynomials. Through their three-term recurrence relation orthogonal polynomials correspond to a Jacobi operator. In particular, spectral analysis of a self-adjoint Jacobi operator leads to an orthogonality measure for the corresponding orthogonal polynomials, i.e., a solution to the moment problem. Indeterminate moment problems correspond to Jacobi operators that are not essentially self-adjoint, and in this case a self-adjoint extension of the Jacobi operator corresponds to a solution of the moment problem, see e.g.~\cite{S}, \cite{Schm}. Instead of looking for self-adjoint extensions of the Jacobi operator on (weighted) $\ell^2(\N)$, it is also useful to look for extensions of the operator to a larger Hilbert space. This method is used in e.g.~\cite{KS03}, \cite{GrK}. A choice of the larger Hilbert space often comes from the interpretation of the operator e.g.~in representation theory.

In this paper we consider an extension of the Jacobi operator for the Al-Salam--Chihara polynomials, and we obtain the spectral decomposition in a similar way as in Koelink and Stokman \cite{KS03}. The Al-Salam--Chihara polynomials are a family of orthogonal polynomials introduced by Al-Salam and Chihara \cite{ASC} that can be expressed as $q$-hypergeometric polynomials. If $q>1$ and under particular conditions on the other parameters, see Askey and Ismail \cite{AI}, they are related to an indeterminate moment problem. In case the moment problem is determinate, the polynomials are orthogonal with respect to a discrete measure. Chihara and Ismail \cite{CI} studied the indeterminate moment problem under certain conditions on the parameters. They obtained the explicit Nevanlinna parametrization, but did not derive explicit solutions. Christiansen and Ismail \cite{ChrI} used the Nevanlinna parametrization to obtain explicit solutions, discrete ones and absolutely continuous ones, of the indeterminate moment problem corresponding to the subfamily of symmetric Al-Salam--Chihara polynomials.  Christiansen and Koelink \cite{CK} found explicit discrete solutions of the symmetric Al-Salam--Chihara moment problem exploiting the fact that the polynomials are eigenfunctions of a second-order $q$-difference operator acting on the variable of the polynomial (whereas the Jacobi operator acts on the degree). The solution we obtain in this paper has an absolutely continuous part and an infinite discrete part. As special cases we also obtain a solution for the symmetric Al-Salam--Chihara moment problem, and for the continuous $q^{-1}$-Laguerre moment problem.

The extension of the Jacobi operator we study in this paper is essentially the $q$-hypergeometric difference operator. The latter is a $q$-analog of the hypergeometric differential operator, whose spectral analysis (on the appropriate Hilbert space) leads the Jacobi function transform, see e.g.~\cite{Koo84}. In a similar way the $q$-hypergeometric difference operator corresponds to the little $q$-Jacobi function transform, see Kakehi, Masuda and Ueno \cite{K95}, \cite{KMU95} and also \cite[Appendix A.2]{KS01}. In this light the integral transform $\mathcal F$ we obtain can be considered as a second version of the little $q$-Jacobi function transform. The little $q$-Jacobi function transform has an interpretations as a spherical transform on the quantum $SU(1,1)$ group. We expect a similar interpretation for the integral transform obtained in this paper.

The organisation of this paper is as follows. In Section \ref{sec:preliminaries} some notations for $q$-hypergeometric functions are introduced and the definition of the Al-Salam--Chihara polynomials is given. In Section \ref{sec:difference operator} the $q$-difference operator $L$ is defined on a family of Hilbert spaces. Using Casorati determinants, which are difference analogs of the Wronskian, it is shown that $L$ with an appropriate domain is self-adjoint. In Section \ref{sec:eigenfunctions} eigenfunctions of $L$ are given in terms of $q$-hypergeometric functions. These eigenfunctions and their Casorati determinants are used to define the Green kernel in Section \ref{sec:spectral decomposition}, which is then used to determine the spectral decomposition of $L$. The discrete spectrum is only determined implicitly, except for one particular choice from the family of Hilbert spaces, where we can explicitly describe the spectrum and the spectral projections. Corresponding to this choice an integral transform $\mathcal F$ is defined in Section \ref{sec:integral transform} that diagonalizes the difference operator $L$, and an explicit inverse transform is obtained. The inverse transform gives rise to orthogonality relations for Al-Salam--Chihara polynomials.

\subsection{Notations and conventions}
Throughout the paper $q \in (0,1)$ is fixed. $\N$ is the set of natural numbers including $0$, and $\T$ is the unit circle in the complex plane. For a set $E \subseteq \R$ we denote by $F(E)$ the vector space of complex-valued functions on $E$.

\section{Preliminaries} \label{sec:preliminaries}
In this section we first introduce standard notations for $q$-hypergeometric functions from \cite{GR04}, state a few useful identities, and then define Al-Salam--Chihara polynomials.

\subsection{$q$-Hypergeometric functions}
The $q$-shifted factorial is defined by
\[
(x;q)_n = \prod_{j=0}^{n-1} 1-xq^{j}, \qquad n \in \N \cup \{\infty\},\quad x \in \C,
\]
where we use the convention that the empty product is equal to 1. Note that $(x;q)_n=0$ for $x \in q^{-\N_{< n}}$.
The following identities for $q$-shifted factorials will be useful later on:
\[
(xq^n;q)_\infty = \frac{ (x;q)_\infty }{ (x;q)_n }, \qquad (xq^{-n};q)_n = (-x)^n q^{-\frac12n(n+1)} (q/x;q)_n.
\]
We also define
\[
(x;q)_{-n} = \frac{ (x;q)_\infty}{(xq^{-n};q)_\infty}, \qquad n \in \N.
\]
The $\te$-function is defined by
\[
\te(x;q)= (x;q)_\infty(q/x;q)_\infty, \quad x \in \C^*.
\]
Note that $\te(x;q)=0$ for $x \in q^\Z$.
We use the following notation for products of $q$-shifted factorials or $\te$-functions
\[
\begin{split}
(x_1,x_2,\ldots,x_k;q)_n &= \prod_{j=1}^k (x_j;q)_n,\\
\te(x_1,x_2,\ldots,x_k;q) &= \prod_{j=1}^k \te(x_j;q).
\end{split}
\]
$\te$ satisfies the identities
\begin{equation} \label{eq:simple theta identities}
\begin{split}
\te(xq^k;q) &= (-x)^{-k} q^{-\frac12k(k-1)} \te(x;q), \\
\te(-x,x;q) &= \te(x^2;q^2),\\
\te(x;q) &= \te(x,qx;q^2).
\end{split}
\end{equation}
We also need the following fundamental $\te$-function identity:
\begin{equation} \label{eq:theta identity}
\theta(xv,x/v,yw,y/w)-\theta(xw,x/w,yv,y/v) = \frac{y}{v} \theta(xy,x/y,vw,v/w).
\end{equation}
The $q$-hypergeometric series $_r\varphi_s$ is defined by
\[
\rphis{r}{s}{a_1,\ldots,a_r}{b_1,\ldots,b_s}{q,x} = \sum_{n=0}^\infty \frac{ (a_1,\ldots,a_r;q)_n }{(q,b_1,\ldots,b_s;q)_n} \left[(-1)^n q^{\hf n(n-1)}\right]^{s-r+1} x^n.
\]
We refer to \cite{GR04} for convergence properties of the series. The $q$-hypergeometric difference equation is given by
\begin{equation} \label{eq:q-hypergeometric difference eq}
(ABx-C)\varphi(qx) + [C+q-(A+B)x] \varphi(x) + (x-q) \varphi(x/q) = 0,
\end{equation}
and has $_2\varphi_1(A,B;C;q,x)$ as a solution. For $q \uparrow 1$ this $q$-difference equation becomes the hypergeometric differential equation.

\subsection{Al-Salam--Chihara polynomials}
The Al-Salam--Chihara polynomials in base $q^{-1}$ are defined by
\begin{equation} \label{def:ASCpol}
P_n(\la;b,c;q^{-1}) = b^{-n}\rphis{3}{2}{q^{n},b\la,b/\la}{bc,0}{q^{-1},q^{-1}},
\end{equation}
which are polynomials of degree $n$ in $\la+\la^{-1}$. They are symmetric in the parameters $b$ and $c$, which follows from transformation formula \cite[(III.11)]{GR04}. The three-term recurrence relation is
\begin{equation} \label{eq:3-term recurrence}
\begin{split}
(\la+\la^{-1}) P_n(\la) = (1-bcq^{-n}) P_{n+1}(\la) + (b+c)q^{-n} P_n(\la) + (1-q^{-n})P_{n-1}(\la),
\end{split}
\end{equation}
where we use the convention $P_{-1} \equiv 0$. For $b+c\in \R$ and $bc\geq0$ the polynomials are orthogonal with respect to a positive measure on $\R$. By \cite[Theorem 3.2]{AI} the moment problem for these polynomials is indeterminate if and only if $b,c \in \R$ with $q < |b/c|<q^{-1}$, or $b=\overline{c}$.

The polynomials $P_n(\la;b,-b;q^{-1})$ are called symmetric Al-Salam--Chihara polynomials, and the polynomials $P_n(\la;q^{\hf(\al+1)},q^{\hf \al};q^{-1})$ with $\al \in \R$, are called continuous $q^{-1}$-Laguerre polynomials. The corresponding moment problems are indeterminate.

\section{The difference operator $L$} \label{sec:difference operator} 
In this section we define the unbounded second-order $q$-difference operator $L$, and the Hilbert space the operator $L$ acts on. The $q$-difference operator is essentially the $q$-hypergeometric difference operator, and, on the given Hilbert space, $L$ can be considered as an extension of the Jacobi operator for the Al-Salam--Chihara polynomials from the previous section.  \\

Let $z \in (q,1]$ and define
\[
I^- = -q^\N, \qquad I^+= zq^\Z, \qquad  I= I^- \cup I^+.
\]
We also set $I_* = I\setminus\{-1\}$. Let $a$ and $s$ be parameters satisfying
$a \in \R$ and $0<a^2<1$, and $s \in \T\setminus\{-1,1\}$ or $s \in \R$ and $q < s^2 < 1$. The condition $s \not\in \{-1,1\}$ is only needed for technical purposes, and can be removed afterwards by continuity in $s$ of the functions involved.
\begin{Def} \label{def:L}
The second order $q$-difference operator $L=L_{a,s}$ on $F(I)$ is given by
\[
\begin{split}
(Lf)(x) &= \frac{1}{a}\left(1+\frac{1}{x}\right)f(x/q) - \frac{ s + s^{-1} }{ax} f(x) + a\left(1+\frac{ 1}{a^2x} \right)f(qx), \qquad x \in I_*,\\
(Lf)(-1) & = \frac{ s+ s^{-1} }{a} f(-1) + a\left(1-\frac{1}{a^2} \right) f(-q).
\end{split}
\]
\end{Def}
\begin{rem}
Define for $x=-q^n \in I^-$
\[
\phi_\la(x) = a^{-n} P_n(\la;s/a,1/sa;q^{-1}).
\]
By \eqref{eq:3-term recurrence} $\phi_\la(x)$ satisfies
\[
\begin{split}
(\la+&\la^{-1}) \phi_\la(-q^n) =\\
&\frac{1}{a}\left(1-q^{-n}\right)\phi_{\la}(-q^{n-1}) - \frac{ s+s^{-1} }{a}q^{-n} \phi_\la(-q^n) + a\left(1-\frac{q^{-n}}{a^2} \right) \phi_{\la}(-q^{n+1}),
\end{split}
\]
so that $\phi_\la$ is an eigenfunction of $L|_{I^-}$. We see that $L$ can be considered as an extension of the Jacobi operator for the Al-Salam--Chihara polynomials. Note also that the parameters $b=s/a$ and $c=1/sa$ satisfy $b=\overline{c}$ in case $s \in \T$, and $q<b/c \leq1$ in case $s \in \R$ with $q<s^2 < 1$, which (by symmetry in $b$ and $c$) corresponds exactly to the conditions under which the moment problem is indeterminate. Note that the families of symmetric Al-Salam--Chihara polynomials and continuous \mbox{$q^{-1}$-Laguerre} polynomials are included (for $s=\pm i$, and $(a,s)=(\pm q^{-\hf\al-\frac14},\pm q^{\frac14} )$, respectively).
\end{rem}

Let $\mathcal L^2_z=\mathcal L^2_z(I, w(x)d_qx)$ be the Hilbert space with inner product
\[
\langle f, g \rangle_z = \int_{-1}^{\infty(z)} f(x) \overline{g(x)} w(x)\, d_qx,
\]
where $w$ is the positive weight function on $I$ given by
\begin{equation} \label{def:weight function}
w(x)= w(x;a;q)= \frac{ (-qx;q)_\infty}{(-a^2x;q)_\infty}.
\end{equation}
Here the $q$-integral is defined by
\[
\int_{-1}^{\infty(z)} f(x)\, d_qx = (1-q)\sum_{n=0}^\infty f(-q^n)q^n + (1-q)\sum_{n=-\infty}^\infty f(zq^n)zq^n.
\]

In order to show that $(L,\mathcal D)$ is self-adjoint, with $\mathcal D \subseteq \mathcal L^2_z$ an appropriate dense domain, we need the following truncated inner product. For $k,l,m\in \N$ and $f,g \in F(I)$ we set
\begin{equation} \label{eq:truncated inner product}
\begin{split}
\langle f, g \rangle_{k,l,m} &= \left( \int_{-1}^{-q^{k+1}} + \int_{zq^{m+1}}^{zq^{-l}} \right) f(x)\overline{g(x)} w(x) d_qx\\
&= (1-q)\sum_{n=0}^k f(-q^n)\overline{g(-q^n)} w(-q^n)q^n \\
& \quad + (1-q)\sum_{n=-l}^m f(zq^n) \overline{g(zq^n)} w(zq^n) zq^n.
\end{split}
\end{equation}
Note that for $f,g \in \mathcal L^2_z$ we have
\begin{equation} \label{eq:limit truncated}
\lim_{k,l,m \rightarrow \infty}\langle f,g \rangle_{k,l,m} = \langle f, g \rangle_z.
\end{equation}
The following result will be useful to establish self-adjointness of $L$.
\begin{lemma} \label{lem:CDet1}
For $k,l,m \in \N$ and $f,g \in F(I)$
\[
\langle Lf, g \rangle_{k,l,m}-\langle f, Lg \rangle_{k,l,m} = D(f,\overline{g})(-q^k) - D(f,\overline{g})(zq^m)+D(f,\overline{g})(zq^{-l-1}),
\]
where $D(f,g) \in F(I)$ is the Casorati determinant given by
\begin{equation} \label{eq:Casorati}
\begin{split}
D(f,g)(x) &= \Big(f(x)g(qx)-f(qx)g(x)\Big) a^{-1}(1-q)(1+a^2x)w(x),
\end{split}
\end{equation}
for $x \in I$.
\end{lemma}
\begin{proof}
We first consider the sum $\sum_{n=0}^k$ of the truncated inner product. We write $f(-q^n)=f_n$ for functions $f$ on $I^-$, then
\[
(Lf)_n = a^{-1}(1-q^{-n}) f_{n-1} + \frac{s+s^{-1}}{a^2q^n}f_n + a(1-a^{-2}q^{-n})f_{n+1}.
\]
Now we have
\[
\begin{split}
\sum_{n=0}^k \Big((Lf)_n \overline{g_n} w_n q^n -& f_k \overline{(Lg)_n} w_n q^n \Big) = \\
& a^{-1}\sum_{n=1}^k f_n \overline{g_{n-1}} (1-q^n)w_n -a^{-1}\sum_{n=1}^k f_{n-1}\overline{g_n}(1-q^n)w_n\\
+& a^{-1}\sum_{n=0}^k f_{n}\overline{g_{n+1}} (1-a^2q^n)w_n- a^{-1}\sum_{n=0}^k f_{n+1}g_n (1-a^2q^n)w_n.
\end{split}
\]
We use
\[
(1-q^{n+1}) w_{n+1} = (1-a^2q^{n}) w_{n}, \qquad n \in \N,
\]
then we see that, after shifting summation indices, the expression above is equal to
\[
\Big(f_k\overline{g_{k+1}}-f_{k+1}\overline{g_{k}}\Big)(1-a^2q^k)a^{-1}w_k,
\]
so we have obtained
\[
\int_{-1}^{-q^{k+1}}\Big((Lf)(x)\overline{g(x)}-f(x)\overline{Lg(x)}\Big)w(x)d_qx = D(f,\overline g)(-q^k).
\]
The sum $\sum_{n=-l}^m$ is treated in the same way, using the identity
\[
(1+zq^{n+1})w(zq^{n+1}) = (1+za^2q^n) w(zq^n), \qquad n \in \Z. \qedhere
\]
\end{proof}

Using the identity
\[
\frac{ (aq^{-n};q)_n }{ (bq^{-n};q)_n} = \frac{ (q/a;q)_n }{ (q/b;q)_n } \left( \frac{a}{b} \right)^n
\]
in the definition \eqref{def:weight function} of the weight function $w$, we see that
\[
w(zq^{-n}) = \frac{ (-zq;q)_\infty }{ (-za^2;q)_\infty } \frac{ (-1/z;q)_n }{ (-q/za^2;q)_n } a^{-2n}q^n,
\]
so that
\begin{equation} \label{eq:asymptotics w}
w(zq^{-n}) = \frac{ \theta(-qz;q)}{\theta(-za^2;q)} a^{-2n} q^n \Big( 1+\mathcal O(q^n)\Big), \qquad n \rightarrow \infty.
\end{equation}
From this we see that if $f \in \mathcal L^2_z$, then
\begin{equation} \label{eq:lim f}
\lim_{n \rightarrow \infty} a^{-n}f(zq^{-n}) =0.
\end{equation}
This enables us to prove the following lemma.
\begin{lemma} \label{lem:CDet2}
For $f,g \in \mathcal L^2_z$, we have
\[
\lim_{n \rightarrow \infty}D(f, \overline{g})(zq^{-n}) = 0.
\]
\end{lemma}
\begin{proof}
From the definition of $w$ we obtain
\[
(1+za^2q^{-n}) w(zq^{-n}) = \frac{ (-zq^{1-n};q)_\infty }{ (-za^2q^{1-n};q)_\infty } = \frac{ (-zq;q)_\infty }{ (-za^2q;q)_\infty } \frac{ (-1/z;q)_n }{ (-1/a^2z;q)_n }a^{-2n},
\]
so that
\[
(1+za^2q^{-n}) w(zq^{-n}) = \frac{ \theta(-zq;q) }{ \theta(-za^2q;q) } a^{-2n}\Big(1+\mathcal O(q^n)\Big), \qquad n \rightarrow \infty.
\]
Now the lemma follows from the definition of the Casorati determinant and the asymptotic behaviour \eqref{eq:lim f} of $f,g \in \mathcal L^2_z$.
\end{proof}

Now we are ready to introduce an appropriate dense domain for the unbounded operator $L$.
Let us define $\mathcal D \subseteq \mathcal L^2_z$  to be the subspace consisting of functions $f \in\mathcal L^2_z$ satisfying the following conditions:
\begin{itemize}
\item $Lf \in \mathcal L^2_z$
\item $f$ is bounded on $-q^{\N+1} \cup z q^\N$
\item $\displaystyle \lim_{n \to \infty}  f(zq^n) - f(-q^n)=0$
\end{itemize}
Note that $\mathcal D$ is dense in $\mathcal L^2_z$, since it contains the finitely supported functions on $I$.
\begin{theorem} \label{thm:L self-adjoint}
The operator $(L,\mathcal D)$ is self-adjoint.
\end{theorem}
\begin{proof}
First we show that $(L,\mathcal D)$ is symmetric.
Let $f,g \in \mathcal D$, and define
\[
u(x) = \frac{a}{(1-q)(1+a^2x)w(x)}, \qquad x \in I,
\]
then from \eqref{eq:Casorati} we obtain
\begin{equation} \label{eq:D-D=}
\begin{split}
 u&(zq^n)D(f,g)(zq^{n}) - u(-q^n) D(f,g)(-q^n) =\\
&  \left[f(zq^n )-f(-q^n)\right] g(zq^{n+1}) - f(zq^{n+1}) \left[ g(zq^n) - g(-q^n) \right] \\
& + f(-q^n) \left[ g(zq^{n+1}) - g(-q^{n+1})\right] - \left[ f(zq^{n+1}) - f(-q^{n+1}) \right] g(-q^n).
\end{split}
\end{equation}
From the conditions on $f$ and $g$ it follows that this tends to $0$ as $n \to \infty$. Then using
\[
\lim_{n \to \infty} u(zq^n) = \lim_{n \to \infty} u(-q^n) = \frac{a}{1-q},
\]
as well as \eqref{eq:limit truncated} and Lemmas \ref{lem:CDet1}, \ref{lem:CDet2} we see that
\[
\langle Lf, g \rangle_z - \langle f, Lg \rangle_z = \lim_{n \to \infty} D(f,g)(-q^n) - D(f,g)(zq^n)=0,
\]
so that $(L, \mathcal D)$ is symmetric.

Since $(L,\mathcal D)$ is symmetric, we have $(L,\mathcal D) \subseteq (L^*, \mathcal D^*)$, where $(L^*, \mathcal D^*)$ is the adjoint of the operator $(L,\mathcal D)$. Here, by definition, $\mathcal D^*$ is the subspace
\[
\mathcal D^* = \{g \in \mathcal L^2_z\ | \ f \mapsto \langle Lf, g \rangle_z \text{ is continuous on } \mathcal D  \}.
\]
We show that $L^* = L|_{\mathcal D^*}$.
Let $f$ be a non-zero function with support at only one point $x \in I$ and let $g \in F(I)$, then we obtain in the same way as above that $\langle Lf,g \rangle_z=\langle f, Lg \rangle_z$. In particular, for $g \in \mathcal D^*$ we then have $\langle f, Lg \rangle_z = \langle f, L^*g \rangle_z$, so $(Lg)(x)=(L^*g)(x)$. This holds for all $x \in I$, hence $L^*=L|_{\mathcal D^*}$.

Finally we need to show that $\mathcal D^* \subseteq \mathcal D$. Let $f \in \mathcal D$ and let $g \in \mathcal D^*$. Using Lemmas \ref{lem:CDet1}, \ref{lem:CDet2} and $L^*=L|_{\mathcal D^*}$ we obtain
\[
\lim_{n \rightarrow \infty} D(f,\overline{g})(-q^n)- D(f,\overline{g})(zq^n) = \langle Lf,g \rangle_z - \langle f,L^*g \rangle_z=0.
\]
Since this holds for all $f \in \mathcal D$, we find using \eqref{eq:D-D=} that $g$ is bounded near $0$, and
\[
\lim_{n \rightarrow \infty}  g(zq^n)-g(-q^n)=0,
\]
hence $g \in \mathcal D$.
\end{proof}

\section{Eigenfunctions of $L$} \label{sec:eigenfunctions}
In this section we consider eigenspaces and eigenfunctions of $L$. The eigenfunctions are initially only defined on a part of $I_*$, and we show how to extend them to functions on $I_*$.  \\

We start with some useful properties of eigenspaces of $L$. For $\mu \in \mathbb C$ we introduce the spaces
\begin{equation} \label{eq:eigenspaces}
\begin{split}
V_\mu^-&=\{ f \in F(I^-_*) \mid Lf=\mu f \},\\
V_\mu^+&=\{ f\in F(I^+)\mid Lf =\mu f \},\\
V_\mu &= \Big\{ f \in F(I_*) \ | \ Lf=\mu f,  \quad  f(zq^n)=o(q^{-n}), \quad f(-q^n) = o(q^{-n})\\
 & \qquad \qquad \qquad   \text{ and } f(zq^n)-f(-q^n) = o(1)\,\text{ for } n \to \infty  \Big\}
\end{split}
\end{equation}
\begin{lemma} \label{lem:eigenspaces}
For $\mu \in \mathbb C$
\begin{enumerate}[(i)]
\item $\dim V_\mu^\pm = 2$.
\item For $f,g \in V_\mu^-$ the Casorati determinant $D(f,g)$ is constant on $I^-_*$.
\item For $f,g \in V_\mu^+$ the Casorati determinant $D(f,g)$ is constant on $I^+$.
\item For $f,g \in V_\mu$ the Casorati determinant $D(f,g)$ is constant on $I_*$.
\end{enumerate}
\end{lemma}
\begin{proof}
For (i) we write $f(tq^n)= f_n$ (with $t=-1$ or $t=z$), then we see that $Lf =\mu f$ gives a recurrence relation of the form $\alpha_n f_{n+1} +\beta_n f_n + \gamma_n f_{n-1} = \mu f_n$, with $\al_n,\ga_n \neq 0$ for all $n$. Solutions to such a recurrence relation are uniquely determined by specifying $f_n$ at two different points $n=l$ and $n=m$. So there are two independent solutions.

The proofs of (ii) and (iii) are the same. We prove (iii). Let $f,g \in F(I^+)$. Using the explicit expressions for the Casorati determinant and the weight function $w$, we find for $x \in I^+$
\[
\begin{split}
(Lf)(x) g(x) &- f(x)(Lg)(x) \\
&= a^{-1}(1+1/x)\Big( f(x/q) g(x)- f(x) g(x/q) \Big) \\
& \quad - a(1+1/a^2x) \Big(f(x)g(qx)- f(qx)g(x) \Big) \\
&= \frac{ (1+1/x) D(f,g)(x/q) }{(1-q)(1+a^2x/q) w(x/q) } - \frac{ a^2(1+1/a^2x) D(f,g)(x) }{ (1-q)(1+a^2x) w(x)}\\
&= \frac{1}{(1-q)x w(x)} \Big( D(f,g)(x/q) - D(f,g)(x) \Big).
\end{split}
\]
Now if $f$ and $g$ satisfy $(Lf)(x)=\mu f(x)$ and $(Lg)(x)=\mu g(x)$, we have $D(f,g)(x/q)= D(f,g)(x)$, hence $D(f,g)$ is constant on $I^+$.

Statement (iv) follows from (ii) and (iii) and the fact that
\[
\lim_{n \to \infty} \Big(D(f,g)(zq^n) - D(f,g)(-q^n)\Big)=0,
\]
which can be obtained by rewriting the difference of the Casorati determinants similar as in the proof of Theorem \ref{thm:L self-adjoint}.
\end{proof}

We can now introduce explicit eigenfunctions of $L$ (in the algebraic sense). Comparing $L$ given in Definition \ref{def:L} to the $q$-hypergeometric difference equation \eqref{eq:q-hypergeometric difference eq} we find that
\begin{equation} \label{def:psi}
\psi_\lambda(x;s) = \psi_\lambda(x;a,s\mvert q) = s^{-n}\rphis{2}{1}{a \lambda/s, a/\lambda s}{q/s^2}{q,-qx},\qquad  |x|<q^{-1},
\end{equation}
with $x=tq^n$, $t \in \{-1,z\}$, is a solution of the eigenvalue equation
\begin{equation} \label{eq:eigenv eq}
Lf=\mu(\lambda)f, \qquad \mu(\lambda)= \lambda+ \lambda^{-1},
\end{equation}
on $I_{*,\leq 1}$.
From $L_{a,s}=L_{a,s^{-1}}$ it follows immediately that $\psi_\lambda(x;a,s^{-1} \mvert q)$ is also a solution of \eqref{eq:eigenv eq} on $I_{*,\leq 1}$. Note that it does not follow from the $q$-hypergeometric difference equation that $\psi_\lambda$ is a solution to \eqref{eq:eigenv eq} for $x=-1$. In fact, we show later on that for generic $\la$ it is not a solution for $x=-1$. Using \eqref{eq:eigenv eq} the solutions $\psi_\lambda(x;s^{\pm 1})$ can be extended to functions on $I_*$ still satisfying \eqref{eq:eigenv eq}. We denote these extensions still by $\psi_\lambda(x;s^{\pm 1})$. Later on we give explicit expressions for the extensions. Note that both functions $\psi_\la(x;s^{\pm 1})$ can be considered as little $q$-Jacobi functions \cite{KMU95}.
\begin{lemma} \label{lem:psi in Vmu}
For $\la \in \C^*$, $\psi_\la(\cdot;s)$ and $\psi_\la(\cdot;s^{-1})$ are in $V_{\mu(\la)}$.
\end{lemma}
\begin{proof}
We already know that $\psi_\la$ is an eigenfunction of $L$, so we only need to check that $\psi_\la$ has the behavior near $0$ as stated in the definition of $V_{\mu(\la)}$.
From \eqref{def:psi} we find for $t \in \{-1,z\}$
\[
\psi_\la(tq^{n};s) = s^{-n}\Big(1+\mathcal O(q^n)\Big), \qquad n \to \infty,
\]
so that $\psi_\la(tq^n;s^{\pm 1}) = o(q^{-n})$ by the conditions on $s$, and
\[
\psi_\la(zq^n;s^{\pm 1}) - \psi_\la(-q^n;s^{\pm 1}) =o(1). \qedhere
\]
\end{proof}

Another solution of \eqref{eq:eigenv eq} is
\[
\Psi_\lambda(x) = \Psi_\lambda(x;a,s\mvert q) = (a\lambda)^{-n} \rphis{2}{1}{a\lambda/s, as\lambda}{q\lambda^2}{q, - \frac{q}{a^2 x} }, \qquad x=zq^n>q/a^2,
\]
and clearly $\Psi_{1/\lambda}(x)$ also satisfies \eqref{eq:eigenv eq}. Both are solutions on $zq^{-\N_{\geq k}}$ for $k$ large enough. For $n \rightarrow \infty$ we have
\begin{equation} \label{eq:asymptotics Psi}
\Psi_\lambda(zq^{-n})= (a\lambda)^n\Big(1+\mathcal O(q^n) \Big),
\end{equation}
so by the asymptotic behavior \eqref{eq:asymptotics w} of the weight function w, for $|\la|<1$ the function $\Psi_\la$ is an $L^2$-function at $\infty$. The solutions $\Psi_{\la^{\pm 1}}$ can be used to give the explicit extension of $\psi_\lambda$ on $I^+$;
\begin{equation} \label{eq:psi=b Psi}
\psi_\lambda(x;s) = b_z(\lambda;s) \Psi_\lambda(x) + b_z(1/\lambda;s)\Psi_{1/\lambda}(x),
\end{equation}
with
\[
b_z(\lambda;s)=b_z(\lambda;a,s \mvert q) = \frac{ (q/as\lambda, a/s\lambda;q)_\infty }{ (1/\lambda^2, q/s^2;q)_\infty } \frac{ \theta(-qaz\lambda/s;q)} {\theta(-qz;q)},
\]
which follows from the three-term transformation formula \cite[(III.32)]{GR04}.

Next we give a solution $\phi_\lambda$ of the eigenvalue equation \eqref{eq:eigenv eq} for all $x \in I$.
\begin{Def} \label{def:phi}
Define
\[
\phi_\lambda(x) = \phi_\lambda(x;a,s\mvert q) = d(\lambda;s) \psi_\lambda(x;a,s\mvert q) + d(\lambda;1/s) \psi_{\lambda}(x;a,1/s \mvert q).
\]
with
\[
d(\lambda;s)=d(\lambda;a,s \mvert q)= \frac{ (as\lambda, as/\lambda;q)_\infty }{ (a^2, s^2;q)_\infty }.
\]
\end{Def}
Note that $\phi_\lambda$ is invariant under $\lambda \leftrightarrow 1/\lambda$ and $s \leftrightarrow 1/s$. Since $\psi_\lambda$ is a solution of \eqref{eq:eigenv eq} for $x \in I_*$, it is clear that $\phi_\lambda$ is also a solution of \eqref{eq:eigenv eq} for $x \in I_*$. We will show that $\phi_\lambda$ is an eigenfunction of $L$ on the whole of $I$, i.e.~including the point $-1$, by identifying $\phi_\la(-q^n)$ with an Al-Salam--Chihara polynomial.
\begin{lemma} \label{lem:phi=ASCpol}
For $n \in \N$ we have
\[
\phi_\la(-q^n) = a^{-n} P_n(\la;s/a,1/sa;q^{-1}).
\]
As a consequence, $L\phi_\la(x)=\mu(\la)\phi_\la(x)$ for all $x \in I$.
\end{lemma}
\begin{proof}
We use the three-term transformation formula for $_3\varphi_2$-functions \cite[(III.34)]{GR04} with $A=q^{-n}$, $n \in \N$. Letting $D \rightarrow \infty$ gives
\[
\begin{split}
\rphis{3}{1}{q^{-n},B,C}{E}{q,\frac{Eq^n}{BC}} &= \frac{ (E/B,E/C;q)_\infty}{(E,E/BC;q)_\infty }\rphis{2}{1}{B,C}{BCq/E}{q,q^{n+1}}\\
 & + \frac{ (B,C;q)_\infty }{ (E, BC/E;q)_\infty } \left( \frac{E}{BC} \right)^n \rphis{2}{1}{E/B, E/C}{Eq/BC}{q,q^{n+1}}.
\end{split}
\]
Using the identity $(A;q^{-1})_n=(1/A;q)_n (-A)^n q^{-\frac12 n(n-1)}$, we find the transformation
\[
\rphis{3}{1}{q^{-n},B,C}{E}{q,\frac{Eq^n}{BC}} = \rphis{3}{2}{q^n,1/B,1/C}{1/E,0}{q^{-1}, q^{-1}}.
\]
From these formulas with $B= a\lambda/s, C= a/\lambda s, E=1/a^2$, we find
\[
\phi_\lambda(-q^n;a,s \mvert q) = s^n \rphis{3}{2}{q^n, s\lambda/a, s/\lambda a }{a^2,0}{q^{-1},q^{-1}},
\]
which we recognize as the Al-Salam--Chihara polynomial $a^{-n} P_n(\la;s/a,1/sa;q^{-1})$, see \eqref{def:ASCpol}. From the recurrence relation for these polynomials it follows that $\phi_\la$ satisfies $(L\phi_\lambda)(x)= \mu(\lambda)\phi_\la(x)$, for $x \in I^-$ (including $-1$), so $\phi_\lambda$ is a solution for \eqref{eq:eigenv eq} on $I$.
\end{proof}
Note that the eigenspace
\[
W_\mu^{-}=\{f\in F(I^-) \mid \ Lf=\mu f\}
\]
is $1$-dimensional, since the value of $f \in W_\mu^-$ at the point $-1$ completely determines $f$ on $I^-$. So $\phi_\lambda$ is the only solution (up to a multiplicative constant) to the eigenvalue equation \eqref{eq:eigenv eq} on $I^-$, hence also the only function in $V_{\mu(\la)}$ which is a solution to \eqref{eq:eigenv eq} on $I$.
\\

We also need the expansion of $\phi_\la$ into $\Psi_{\la^{\pm 1}}$ and we need to determine an explicit expression for the Casorati determinant $D(\phi_\la,\Psi_\la)$, which we need later on to describe the resolvent for $L$. Determining the expansion is a straightforward calculation.
\begin{lemma} \label{lem:c-function expansion}
For $x \in I^+$,
\[
\phi_\la(x) = c_z(\la) \Psi_\la(x) + c_z(\la^{-1}) \Psi_{\la^{-1}}(x),
\]
with
\[
\begin{split}
c_z(\la) &= c_z(\la;a,s \mvert q) \\
&= \frac{(as/\la, a/s \la;q)_\infty}{(a^2,\la^{-2};q)_\infty \te(-qz;q)} \Big( \frac{\te(as\la, -qaz\la/s;q)}{\te(s^2;q)} + \frac{\te(a\la/s, -qasz\la;q)}{\te(s^{-2};q)} \Big).
\end{split}
\]
\end{lemma}
\begin{proof}
From \eqref{eq:psi=b Psi} and Definition \ref{def:phi} we find
\[
\phi_\la = c_z(\la) \Psi_\la + c_z(\la^{-1}) \Psi_{\la^{-1}}, \qquad \text{on } I^+,
\]
with
\[
c_z(\la) = d(\la;a,s \mvert q) b_z(\la;a,s \mvert q) + d(\la;a,1/s \mvert q) b_z(\la;a,1/s \mvert q).
\]
From the explicit expression for $d(\la)$ and $b_z(\la)$ we obtain the expression for $c_z(\la)$.
\end{proof}

\begin{lemma} \label{lem:Casorati determinants}
For $x \in I^+$,
\[
D\big(\psi_\lambda(\cdot;s), \psi_\lambda(\cdot;1/s)\big)(x) = a(1-q)(s^{-1}-s)
\]
and
\[
D(\phi_\lambda, \Psi_\lambda)(x) =K_z\, c_z(\la^{-1}) (\lambda^{-1}-\lambda),
\]
with
\[
K_z= (1-q)\frac{ \theta(-z;q) }{\theta(-za^2;q)}.
\]
\end{lemma}
\begin{proof}
The Casorati determinant $D(\Psi_{\lambda^{-1}},\Psi_\lambda)$ is constant on $I^+$, so we may determine it by letting $x \rightarrow \infty$.
Now from the asymptotic behavior \eqref{eq:asymptotics Psi} of $\Psi_\lambda(zq^{-k})$ when $k \rightarrow \infty$, the definition \eqref{eq:Casorati} of the Casorati determinant, and the asymptotic behavior of $(1+za^2q^{-k})w(zq^{-k})$ for $k \rightarrow \infty$, see Lemma \ref{lem:CDet2}, we obtain
\[
D(\Psi_{\lambda^{-1}},\Psi_\lambda)(x)  = K_z (\lambda^{-1} - \lambda).
\]
With this result we can compute $D\big(\psi_{\lambda}(\cdot;s), \Psi_\lambda \big)$.
First expand $\psi_\lambda(x;s)$ in terms of $\Psi_{\lambda^{\pm 1}}(x)$ using \eqref{eq:psi=b Psi}, then we find
\[
D(\psi_{\lambda}(\cdot;s), \Psi_\lambda)(x) = b_z(1/\lambda;s) D(\Psi_{1/\lambda},\Psi_\lambda)(x) = K_z b_z(1/\lambda;s) (\lambda-\lambda^{-1}).
\]
This leads to
\[
\begin{split}
D\big(\psi_\lambda(\cdot;s), &\psi_\lambda(\cdot;1/s)\big)(x)\\
&= b_z(\lambda;1/s) D( \psi_\lambda(\cdot;s), \Psi_\lambda)(x) + b_z(1/\lambda;1/s) D( \psi_\lambda(\cdot;s), \Psi_{1/\lambda})(x) \\
& = K_z \Big(  b_z(1/\lambda;s)b_z(\lambda;1/s) - b_z(\lambda;s)b_z(1/\lambda;1/s)\Big) (\lambda - \lambda^{-1}).
\end{split}
\]
Writing this out and using the fundamental $\theta$-function identity \eqref{eq:theta identity}
with
\[
(x,y,v,w)= (a, -az, \lambda/s, \lambda s)
\]
we find
\[
D\big(\psi_\lambda(\cdot;s), \psi_\lambda(\cdot;1/s)\big)(x) = \frac{s K_z(\lambda - \lambda^{-1})}{az\la} \frac{ \theta(-a^2z, -1/z, \lambda^2, 1/s^2;q) }{(\lambda^{\pm 2}, qs^{\pm 2};q)_\infty \te(-qz;q)^2}.
\]
This expression simplifies using identities for $\te$-functions.

Finally, using the $c$-function expansion from Lemma \ref{lem:c-function expansion} we have
\[
D(\phi_\la,\Psi_\la) = c_z(\la^{-1}) D(\Psi_{\la^{-1}},\Psi_\la) = K_z\, c_z(\la^{-1})(\la^{-1} - \la). \qedhere
\]
\end{proof}
Note that so far we only have expressions for the Casorati determinants on $I^+$. These expressions are actually valid on $I_*$, which follows from the following result.
\begin{prop} \label{prop:basis}
The set $\{\psi_\lambda(\cdot;s), \psi_\lambda(\cdot;1/s)\}$ is a basis for $V_{\mu(\lambda)}$.
\end{prop}
\begin{proof}
From Lemma \ref{lem:psi in Vmu} we know that $\psi_\la(\cdot;s^{\pm 1}) \in V_{\mu(\la)}$. Then by Lemma \ref{lem:eigenspaces} the Casorati determinant $D\big(\psi_\lambda(\cdot;s), \psi_\lambda(\cdot;1/s)\big)$ is constant on $I_*$, so the value from Lemma \ref{lem:Casorati determinants} is valid on $I_*$. It is nonzero since we assumed $s \neq \pm 1$, so $\psi_\la(\cdot;s)$ and $\psi_\la(\cdot;1/s)$ are linearly independent, and since $\dim V_{\mu(\la)}=2$ they form a basis.
\end{proof}

Now $\Psi_\lambda$ extends to a function on $I_*$ by expanding it in terms of $\psi_{\lambda}(x;s^{\pm 1})$. This can be done explicitly, but we do not need the explicit expansion. Then $\Psi_\la, \phi_\la \in V_{\mu(\la)}$, so as a consequence we obtain the following result.
\begin{cor} \label{cor:Dlambda}
On $I_*$,
\[
D(\phi_\la,\Psi_\la) = K_z\, c_z(\la^{-1})(\la^{-1}-\la), \qquad \la \in \C^*.
\]
.
\end{cor}

\section{The spectral decomposition of $L$} \label{sec:spectral decomposition}
In this section we determine the spectral decomposition for $L$ as a self-adjoint operator on $\mathcal L^2_z$. For $z=1$ the spectral projections are completely explicit.\\

To determine the spectral measure $E$ for the self-adjoint operator $(L,\mathcal D)$ we use the following formula, see \cite[Theorem XII.2.10]{DS63},
\begin{equation} \label{eq:Stieltjes-Perron}
\langle E(a,b)f, g \rangle_z = \lim_{\de \downarrow 0} \lim_{\eps \downarrow 0} \frac{1}{2\pi i} \int_{a + \de}^{b-\de} \Big( \langle R_{\ga+i\eps} f,g \rangle_{z} - \langle R_{\ga-i\eps} f,g \rangle_z \Big) d\ga,
\end{equation}
for $a<b$ and $f,g \in \mathcal L^2_z$. Here $R_\ga=(L-\ga)^{-1}$ is the resolvent for $L$. $R_{\mu(\la)}$ can be expressed using the Green kernel $G_\la \in F(I\times I)$ given by
\[
G_\la(x,y) =
\begin{cases}
\dfrac{\phi_\la(x) \Psi_\la(y)}{D_\la}, & x \leq y, \\ \\
\dfrac{\phi_\la(y) \Psi_\la(x)}{D_\la}, & x > y,
\end{cases}
\]
where $D_\la=D(\phi_\la,\Psi_\la)$ (see Corollary \ref{cor:Dlambda}) and we assume $\la \in \mathbb D \setminus (-1,1)$, where $\mathbb D$ is the open unit disc. Note that $G_\la(\cdot,y) \in \mathcal L^2_z$ for $y \in I$. Now the resolvent is given by
\[
(R_{\mu(\la)} f)(y) = \langle f, G_\la(\cdot,y) \rangle_z,
\]
so that for $f,g \in \mathcal L^2_z$
\[
\begin{split}
\langle &R_{\mu(\la)} f, g \rangle_z \\
&= \iint_{I\times I} f(x) \overline{g(y)} G_\la(x,y) \, w(x) w(y) \, d_qx d_qy \\
& = \iint\limits_{\substack{(x,y) \in I \times I \\ x \leq y}} \frac{\phi_\la(x) \Psi_\la(y) }{D_\la}\Big(f(x) \overline{g(y)}+ f(y) \overline{g(x)}\Big) \big(1-\tfrac12 \de_{x,y}\big) w(x)w(y) \, d_qx \, d_qy.
\end{split}
\]
The proof that $R_{\mu(\la)}$ is indeed the resolvent operator is identical to the proof of \cite[Proposition 6.1]{KS03}.

Let $\mu^{-1}:\C\setminus \R \to \mathbb D\setminus (-1,1)$  be the inverse of $\mu|_{\mathbb D \setminus (-1,1)}$. Let $\ga \in (-2,2)$, then $\ga= \mu(e^{i\psi})$ for a unique $\psi \in (0,\pi)$, and in this case
\[
\lim_{\eps \downarrow 0}\mu^{-1}(\ga\pm i\eps) = e^{\mp i\psi}.
\]
Furthermore, for $\ga \in \R\setminus[-2,2]$ we have $\ga =\mu(\la)$ for a unique $\la \in (-1,1)$, and in this case
\[
\lim_{\eps \downarrow 0} \mu^{-1}( \ga \pm \eps ) = \la.
\]
We see that we need to distinguish between spectrum in $(-2,2)$ and in $\R\setminus[-2,2]$.
\begin{prop} \label{prop:E continuous}
Let $0< \psi_2 < \psi_1 < \pi$, and let $\mu_1=\mu(e^{i\psi_1})$ and $\mu_2 = \mu(e^{i\psi_2})$. Then, for $f,g \in \mathcal L^2_z$,
\[
\langle E(\mu_1,\mu_2)f,g \rangle_z = \frac{1}{2\pi K_z} \int_{\psi_2}^{\psi_1} \langle f, \phi_{e^{i\psi}} \rangle_z  \langle \phi_{e^{i\psi}},g \rangle_z \frac{d\psi}{|c_z(e^{i\psi})|^2}.
\]
\end{prop}
\begin{proof}
For $\ga = \mu(e^{i\psi})$ with $0 < \psi < \pi$, we have
\[
\begin{split}
\lim_{\eps \downarrow 0} & \left(\frac{\phi_{\mu^{-1}(\ga+i\eps)}(x) \Psi_{\mu^{-1}(\ga+i\eps)}(y) }{D_{\mu^{-1}(\ga+i\eps)}} - \frac{\phi_{\mu^{-1}(\ga-i\eps)}(x) \Psi_{\mu^{-1}(\ga-i\eps)}(y) }{D_{\mu^{-1}(\ga-i\eps)}} \right)\\
&=\frac{\phi_{e^{-i\psi}}(x) \Psi_{e^{-i\psi}}(y) }{D_{e^{-i\psi}}} - \frac{\phi_{e^{i\psi}}(x) \Psi_{e^{i\psi}}(y) }{D_{e^{i\psi}}} \\
&=  \frac{\phi_{e^{i\psi}}(x) \left(c_z(e^{-i\psi}) \Psi_{e^{-i\psi}}(y) + c_z(e^{i\psi}) \Psi_{e^{i\psi}}(y) \right)}{2iK_z\sin(\psi) \, c_z(e^{i\psi})c_z(e^{-i\psi})}\\
& = \frac{\phi_{e^{i\psi}}(x)\phi_{e^{i\psi}}(y) }{2iK_z\sin(\psi) \, c_z(e^{i\psi})c_z(e^{-i\psi})},
\end{split}
\]
where we use $\phi_\la=\phi_{\la^{-1}}$ and the explicit expression for $D_\la$. Then the result follows from the inversion formula \eqref{eq:Stieltjes-Perron} and $d\mu(e^{i\psi}) = -2 \sin(\psi) d\psi$.
\end{proof}

From the $c$-function expansion in Lemma \ref{lem:c-function expansion}, the asymptotic behavior of $\Psi_\la$ \eqref{eq:asymptotics Psi} and the weight $w$ \eqref{eq:asymptotics w}, it follows that $\phi_\la \not\in \mathcal L^2_z$ for $\la \in \T$, so $(-2,2)$ is contained in the continuous spectrum of $L$. The points $-2$ and $2$ are also in the continuous spectrum, which follows in the same way as in \cite[Lemma 7.1]{KS03}.
\\

Next we consider the spectrum in $\R\setminus[-2,2]$. Assume $\ga \in \R\setminus[-2,2]$, then $\ga$ only contributes to the spectral measure if $\ga=\mu(\la)$ with $\la$ a simple pole of \mbox{$\la \mapsto G_\la$}, i.e., $c_z(\la^{-1})=0$. Let us assume that all real zeros of $\la \mapsto c_z(\la^{-1})$ are simple, and let $S_z$ be the set of all those zeros. For $\ga \in S_z$ we write \mbox{$E(\{\ga\}) = E(a,b)$} where $(a,b)$ is an interval disjoint from $[-2,2]$ such that $(a,b) \cap S_z = \{\ga\}$. Note that $E(a,b)=0$ if $(a,b)\cap [-2,2]=\emptyset$ and $(a,b)\cap S_z = \emptyset$.
\begin{prop} \label{prop:E discrete}
Let $f,g \in \mathcal L^2$, and let $\ga=\mu(\la)$ with $\la \in S_z$, then
\[
\langle E(\{\ga\}) f, g \rangle_z = \frac{1}{K_z} \langle f,\phi_\la \rangle_z \langle \phi_\la, g \rangle_z \Res{\la'=\la} \left( \frac{1}{\la'c_z(\la') c_z(\la'^{-1})} \right).
\]
\end{prop}
\begin{proof}
Let $\mathcal C$ be a clockwise oriented contour encircling $\ga$ once, such that no other points in $S_z$ are enclosed by $\mathcal C$. Then by the residue theorem
\[
\begin{split}
\langle E(\{\ga\}) f, g \rangle_z
& = \frac{1}{2\pi i} \int_{\mathcal C} \langle R_\ga f, g \rangle \, d\ga \\
&  = -\iint\limits_{\substack{(x,y) \in I \times I \\ x \leq y}} \phi_\la(x) \Psi_\la(y)  \Res{\la'=\la} \left( \frac{1-1/\la'^2}{D_{\la'}} \right) \\
& \qquad \qquad \times \Big(f(x) \overline{g(y)}+ f(y) \overline{g(x)}\Big) \big(1-\tfrac12 \de_{x,y}\big) w(x)w(y) \, d_qx \, d_qy \\
& = \frac{1}{K_z}\iint\limits_{\substack{(x,y) \in I \times I \\ x \leq y}} \phi_\la(x) \phi_\la(y)  \Res{\la'=\la} \left( \frac{1}{\la'c_z(\la') c_z(\la'^{-1})} \right) \\
& \qquad \qquad \times \Big(f(x) \overline{g(y)}+ f(y) \overline{g(x)}\Big) \big(1-\tfrac12 \de_{x,y}\big) w(x)w(y) \, d_qx \, d_qy.
\end{split}
\]
Here we used $\phi_\la = c_z(\la) \Phi_\la$, since $c_z(\la^{-1})=0$. Symmetrizing the double $q$-integral gives the result.
\end{proof}

For $\la \in S_z$ we have $\phi_\la \in \mathcal L^2_z$, so $\mu(S_z)$ is the discrete spectrum of $L$. Eigenfunctions (in the analytic sense) of self-adjoint operators are mutually orthogonal, so we have the following orthogonality relations.
\begin{cor} \label{cor:orthogonality relations}
For $\la,\la' \in S_z$,
\[
\langle \phi_\la,\phi_{\la'} \rangle_z = \de_{\la,\la'}\frac{K_z}{ \Res{\la'=\la} \left( \frac{1}{\la'c_z(\la') c_z(\la'^{-1})} \right)}.
\]
\end{cor}
\begin{proof}
Let $\la \in S_z$. We only have to compute the norm of $\phi_\la$. Take $f=g=\phi_\la$ and $\ga=\mu(\la)$ in Proposition \ref{prop:E discrete}, then
\[
\langle \phi_\la,\phi_{\la} \rangle_z = \langle E(\{\mu(\la)\}) \phi_\la,\phi_\la\rangle_z = \frac{1}{K_z} \Res{\la'=\la} \left( \frac{1}{\la'c_z(\la') c_z(\la'^{-1})} \right) \langle \phi_\la,\phi_\la \rangle_z^2,
\]
which gives the result.
\end{proof}

So far the discrete spectrum $\mu(S_z)$ is only defined implicitly, as $S_z$ is the set of zeros of $c_z(\la^{-1})$ inside $(-1,1)$. To make this explicit we need to solve the equation
\[
(as\la, a\la/s ;q)_\infty \Big( \frac{\te(as/\la, -qaz/\la s;q)}{\te(s^2;q)} + \frac{\te(a/\la s, -qasz/\la;q)}{\te(s^{-2};q)}\Big)=0,
\]
see Lemma \ref{lem:c-function expansion}. The factor $(as\la, a\la/s;q)_\infty$ has $\frac{1}{as}q^{-\N} \cup \frac{s}{a} q^{-\N}$ as the set of zeros. But it does not seem possible to give explicit expressions for the zeros of the second factor, the sum of $\te$-functions, except in case $z=1$.
\begin{lemma} \label{lem:c-function}
For $z=1$,
\[
c_1(\la) = \frac{(as/\la, a/s \la;q)_\infty \te(a^2\la^2q;q^2)}{ (a^2,\la^{-2};q)_\infty \te(qs^2;q^2)}.
\]
As a consequence, $c_1(\la^{-1})=0$ if and only if $\la\in \Ga$ with
\[
\Ga = \frac{1}{as}q^{-\N} \cup \frac{s}{a} q^{-\N} \cup a q^{\hf+\Z} \cup (-a)q^{\hf + \Z}.
\]
\end{lemma}
\begin{proof}
For $z=1$ we can simplify the expression for $c_z(\la)$ using identities for $\te$-functions from Section \ref{sec:preliminaries}. We have
\[
\te(-qaz \la s^{\pm 1};q)= \te(-qa\la s^{\pm 1};q ) = \frac{1}{a\la s^{\pm 1}} \te(-a\la s^{\pm 1};q),
\]
so that
\[
c_1(\la) = \frac{s}{a\la} \frac{(as/\la, a/s \la;q)_\infty}{ (a^2,\la^{-2};q)_\infty \te(-q,s^2;q)} \Big( \te(as\la, -a\la/s;q) - \te(a\la/s, -as\la;q) \Big).
\]
Use the fundamental $\te$-function identity \eqref{eq:theta identity} with
\[
(x,y,v,w) = (a\la s^\frac12, a \la s^{-\frac12}, s^\frac12,  -s^{-\frac12})
\]
for a fixed root $s^\frac12$ of $s$, to obtain
\[
\begin{split}
\te(as\la, -a\la/s;q) - \te(a\la/s, -as\la;q) &= \frac{ a\la}{s} \frac{ \te(a^2 \la^2, s, -1, -s;q) }{\te(a\la,-a\la;q)}.
\end{split}
\]
Then simplifying the remaining expressions using \eqref{eq:simple theta identities} gives the result.
\end{proof}
Recall that we assumed $0 < a^2 <1$, and $s \in \T$ or $s \in \R$ with $q<s^2<1$. Using \ref{lem:c-function} we can read off the zeros of $c_1(\la^{-1})$ that are inside $(-1,1)$. This gives the following result for the spectrum.
\begin{prop} \label{prop:spectrum}
The spectrum of the self-adjoint $q$-difference operator $L$ on $\mathcal L^2_1$ is
\[
[-2,2]\cup \mu(S_1) ,
\]
where for $s \in \T\setminus\{-1,1\}$ or $s \in \R$ such that $|s/a| \geq 1$
\[
S_1 = \left\{\pm aq^{m+\frac12}  \mid m \in \Z \text{ such that } -1<aq^{m+\frac12} < 1 \right\},
\]
and for $s \in \R$ such that $|s/a|<1$
\[
S_1 = \left\{ \frac{s}{a} \right\} \cup \left\{\pm aq^{m+\frac12}  \mid m \in \Z \text{ such that } -1<aq^{m+\frac12} < 1 \right\}.
\]
\end{prop}

\section{The integral transform} \label{sec:integral transform}
In this section we arrive at the main result. We define an integral transform $\mathcal F$ that can be considered as a $q$-analog of the Jacobi function transform. We show that $\mathcal F$ is unitary and we determine the inverse of $\mathcal F$. As a consequence we find orthogonality relations for $q^{-1}$-Al-Salam--Chihara polynomials, and hence we have a solution of the corresponding moment problem. Throughout this section we assume $z=1$ and omit all the subscripts $z$; in particular, $\mathcal L^2=\mathcal L^2_1$ and $c(z)=c_1(z)$.\\

We define the integral transform $\mathcal F$ as follows.
\begin{Def}
Let $\mathcal D_0 \subseteq \mathcal L^2$ be the subset consisting of finitely supported functions. $\mathcal F:\mathcal D_0 \to F(\T \cup S)$ is given by \[
(\mathcal F f)(\la) = \int_{-1}^{\infty(1)} f(x) \phi_\la(x) w(x) \, d_qx, \qquad \la \in \T \cup S.
\]
\end{Def}

Let $\nu$ be the measure defined by
\[
\int f(\la) \, d\nu(\la) = \frac{1}{4 K\pi i} \int_\T f(\la) W(\la) \frac{ d\la}{\la} + \frac{1}{K}\sum_{\la \in S} f(\la) \hat W(\la),
\]
where $\T$ is oriented in the counter-clockwise direction,
\[
K=(1-q)\frac{ \te(-1;q) }{\te(-a^2;q)},
\]
\[
W(\la) = \frac{1}{|c(\la)|^2} =
\left|\frac{(as/\la, a/s \la;q)_\infty \te(a^2\la^2q;q^2)}{ (a^2,\la^{-2};q)_\infty \te(qs^2;q^2)} \right|^2,\qquad \la \in \T,
\]
and
\[
\hat W(\la) = \Res{\la'=\la}\left( \frac{1}{\la' c(\la') c(\la'^{-1}) } \right), \qquad \la \in S,
\]
with $S=S_1$ is given in Proposition \ref{prop:spectrum}. The residues can be computed explicitly, which gives the following expressions: for $\pm aq^{m+\hf} \in S$, 
\[
\begin{split}
\hat W(\pm aq^{m+\hf}) & =
 \frac{(a^2,a^2, a^2q, a^{-2}q^{-1};q)_\infty \te(qs^2,q/s^2;q^2)}{2(q^2;q^2)_\infty^2 (\pm a^2 q^\hf s,  \pm a^2 q^\hf / s, \pm q^{-\hf}s, \pm q^{-\hf}/s ;q)_\infty \te(a^4 q^2;q^2) } \\
& \quad \times \frac{ 1- a^2 q^{2m+1} }{1-a^2 q}  \frac{(\pm a^2 q^\hf s, \pm a^2 q^\hf /s;q)_m  }{ (\pm q^{\frac32} s, \pm q^{\frac32}/s;q)_m } q^{m(m+1)},
\end{split}
\]
and if $s/a \in S$,
\[
\begin{split}
\hat W(s/a) =  \frac{ (a^2, s^2/a^2;q)_\infty \te(s^2q;q^2) }{(q,s^2;q)_\infty \te(a^4q/s^2;q^2) }.
\end{split}
\]
Let $\mathcal H$ be the Hilbert space consisting of functions $f$ satisfying $f(\la)=f(\la^{-1})$ $\nu$-almost everywhere, with inner product
\[
\langle f,g \rangle_{\mathcal H} = \int f(\la) \overline{g(\la)}\, d\nu(\la).
\]
Define $\mathcal G:\mathcal H \to F(I)$ by
\[
(\mathcal G g)(x) = \langle g, \phi_{\bullet}(x) \rangle_{\mathcal H}, \qquad x \in I.
\]
We can now formulate the main result.
\begin{theorem} \label{thm:main}
$\mathcal F$ extends uniquely to a unitary operator $\mathcal F:\mathcal L^2 \to \mathcal H$, with inverse $\mathcal G$.
\end{theorem}
As a consequence we have orthogonality relations for the functions $\phi_\bullet(x)$ with respect to the measure $\nu$. Since $\phi_\la(-q^{n})$ is an Al-Salam--Chihara polynomial by Lemma \ref{lem:phi=ASCpol}, $\nu$ is a solution of the corresponding indeterminate moment problem. Note that the polynomials are not dense in $\mathcal H$, so the measure $\nu$ is not an $N$-extremal solution.
\begin{cor} \label{cor:solution moment problem}
The set $\{\phi_\bullet(x) \mid x \in I\}$ is an orthogonal basis for $\mathcal H$, with
\[
\langle \phi_\bullet(x), \phi_\bullet(y) \rangle_{\mathcal H} = \de_{x,y} \frac{1}{|x| w(x) }.
\]
In particular, the $q^{-1}$-Al-Salam--Chihara polynomials $P_n(\la) = P_n(\la;s/a,1/sa;q^{-1})$ satisfy
\[
\int P_n(\la) P_{n'}(\la) \, d\nu(\la) = \de_{n,n'}\frac{a^{2n}}{q^n w(-q^n)}.
\]
\end{cor}

\begin{rem}
The kernel $\phi_\la$ defined in Definition \ref{def:phi} is a linear combination of the functions $\psi_\la(\cdot;s^{\pm 1})$, which are essentially little $q$-Jacobi functions. Furthermore, $\mathcal F$ diagonalizes the $q$-hypergeometric difference operator $L$ considered as an unbounded operator on $\mathcal L^2$, i.e.
\[
\mathcal F \, L \, \mathcal F^{-1} = M_\mu,
\]
where $M_\mu$ is multiplication by $\mu(\,\cdot\,)$ on $\mathcal H$. So we may consider $\mathcal F$ as another little $q$-Jacobi function transform.
\end{rem}

\begin{rem}
As already mentioned in Section 3 the symmetric Al-Salam--Chihara polynomials and the $q$-Laguerre polynomials are special cases of the Al-Salam--Chihara polynomials we consider, so we also obtained a solution for the corresponding moment problems. These solutions seem to be new. We can also (formally) obtain a solution of the Al-Salam--Carlitz II moment problem.
Assume $|s/a|\geq 1$, then from Proposition \ref{prop:spectrum} we see that the operator $aL_{a,s}$ has spectrum
\[
[-2a,2a] \cup \big\{ \pm (a^2 q^{m+\hf} + q^{-m-\hf}) \mid m \in \Z \text{ such that } -1 < aq^{m+\hf} < 1\big\}.
\]
For $a \to 0$ the continuous spectrum shrinks to $\{0\}$ and the discrete spectrum becomes $\{ \pm q^{-m-\hf}\mid m \in \Z\}$. From the recurrence relation for $a^{-n} P_n(\pm a q^{m+\hf})$ we find that the polynomials $p_n(\pm q^{m+\hf} ) = \lim_{a \to 0} a^{-n}P_n(\pm a q^{m+\hf};s/a,1/as;q^{-1})$ satisfy
\[
\pm q^{-m-\hf} p_n = - q^{-n} p_{n+1} + (s+s^{-1}) q^{-n} p_n+(1-q^{-n}) p_{n-1}.
\]
Comparing this to the recurrence relation for the Al-Salam--Carlitz II polynomials $V_n^{(a)}(x;q)$ \cite[\S14.25]{KLS}, we find that $p_n(\la) = (-s)^n q^{\hf n(n-1)}V_n^{(s^{-2})}(\la;q)$. Letting $a \to 0$ in the orthogonality relations for $a^{-n}P_n(\pm a q^{m+\hf};s/a,1/as;q^{-1})$ we find that the polynomials $p_n$ satisfy
\[
\begin{split}
\sum_{m \in \Z} & p_{n}(q^{-m-\hf}) p_{n'}(q^{-m-\hf}) \frac{q^{m(m+1)}}{(q^{\frac32} s,q^{\frac32}/s;q)_m } \\
& \quad + p_{n}(-q^{-m-\hf}) p_{n'}(-q^{-m-\hf}) \frac{q^{m(m+1)}}{(-q^{\frac32} s,-q^{\frac32}/s;q)_m } = \de_{n,n'} N_n,
\end{split}
\]
where $N_n$ can be determined explicitly. This is a special case of the solutions found in \cite{Gr}. 
\end{rem}

Let us now turn to the proof of Theorem \ref{thm:main}, which takes several steps. From the results of the previous section we obtain the following result.
\begin{prop} \label{prop:isometry}
The map $\mathcal F$ extends uniquely to an isometry $\mathcal F:\mathcal L^2\to \mathcal H$.
\end{prop} \label{prop:isometry}
\begin{proof}
Let $f_1,f_2 \in \mathcal D_0$, then by Propositions \ref{prop:E continuous}, \ref{prop:E discrete}
and the definition of the inner product $\langle \cdot, \cdot \rangle_{\mathcal H}$,
\[
\langle f_1,f_2 \rangle = \langle E(\R) f_1,f_2 \rangle = \langle \mathcal Ff_1,\mathcal Ff_2 \rangle_{\mathcal H},
\]
from which it follows that $\mathcal F$ extends to an isometry $\mathcal L^2\to \mathcal H$.
\end{proof}

We show that $\mathcal F$ is unitary by showing that $\mathcal G$ is indeed the inverse of $\mathcal F$. Note that $\phi_\la(x) = (\mathcal F d_x)(\la)$, where $d_x \in \mathcal L^2$ is given by $d_x(y) = \frac{\de_{x,y}}{w(x)|x|}$. So $\phi_{\bullet}(x)  \in \mathcal H$ by Proposition \ref{prop:isometry}, and we see that $\mathcal Gg$ exists for all $g \in \mathcal H$. Using the functions $d_x$ it is easy to show that $\mathcal G$ is a left-inverse of $\mathcal F$.
\begin{lemma} \label{lem:GF=id}
$\mathcal G \mathcal F = \mathrm{id}_{\mathcal L^2}$.
\end{lemma}
\begin{proof}
Let $f \in \mathcal L^2$ and $x \in I$. Using $(\mathcal F d_x)(\la) = \phi_\la(x)$ and Proposition \ref{prop:isometry} we have
\[
(\mathcal G \mathcal F f)(x) = \langle \mathcal Ff , \mathcal Fd_x \rangle_{\mathcal H} = \langle f,d_x \rangle = f(x). \qedhere
\]
\end{proof}

It requires more work to show that $\mathcal G$ is also a right inverse of $\mathcal F$. We use a classical method \cite{BM67}, \cite{Go65}, which is essentially approximating with the Fourier transform. We need the truncated inner product, see \eqref{eq:truncated inner product}, of $\phi_\la$ and $\phi_{\la'}$ with $\la \neq \la'$. We first derive a results about these inner products that will be useful later on.
\begin{lemma} \label{lem:properties truncated inner product}
For $l \in \N$ and $\la, \la' \in \C^*$ with $\mu(\la) \neq \mu(\la')$, the limit
\[
\langle \phi_\la,\phi_{\la'} \rangle_l = \lim_{k \to \infty} \langle \phi_\la, \phi_{\la'} \rangle_{k,l,k}
\]
exists and for $l \to \infty$,
\[
\langle \phi_\la,\phi_{\la'} \rangle_l = K\sum_{\eps,\eta \in \{-1,1\}} \frac{(\la^{\eps} - \la'^\eta)(\la^{\eps} \la'^{\eta} )^l c(\la^{\eps}) c(\la'^{\eta})}{\mu(\la)-\mu(\la')}\Big(1+\mathcal O(q^l)\Big).
\]
\end{lemma}
\begin{proof}
From Lemma \ref{lem:CDet1} we obtain
\[
\begin{split}
\big(\mu(\la)-\mu(\la') \big) & \langle \phi_\la,\phi_{\la'} \rangle_{k,l,m} = \\
&D(\phi_\la,\phi_{\la'})(-q^k) - D(\phi_\la,\phi_{\la'})(q^m) + D(\phi_\la,\phi_{\la'})(q^{-l-1}).
\end{split}
\]
Using $\phi_\la \in V_{\mu(\la)}$, see \eqref{eq:eigenspaces}, we obtain from the expression \eqref{eq:Casorati} for the Casorati determinant that
\[
\lim_{k \to \infty} D(\phi_\la,\phi_{\la'})(-q^k) - D(\phi_\la,\phi_{\la'})(q^k) = 0,
\]
which shows that
\[
\langle \phi_\la,\phi_{\la'} \rangle_l = \frac{ D(\phi_\la,\phi_{\la'})(q^{-l-1})}{\mu(\la)-\mu(\la')}.
\]
The asymptotic behavior of this Casorati determinant can be obtained, using the $c$-function expansion in Lemma \ref{lem:c-function expansion}, from the asymptotic behavior of the Casorati determinant
$D(\Psi_\la,\Psi_{\la'})(q^{-l-1})$. From \eqref{eq:Casorati} and the asymptotic behavior \eqref{eq:asymptotics Psi} and \eqref{eq:asymptotics w} of $\Psi_\la$ and $w$, we obtain
\[
D(\Psi_\la,\Psi_{\la'})(q^{-l-1}) = K(\la-\la')(\la \la')^l \big( 1 + \mathcal O(q^l) \big), \qquad l \to \infty,
\]
from which the result follows.
\end{proof}

Next we show that $\langle \phi_\la,\phi_{\la'} \rangle_l$ has a reproducing property similar to the Dirichlet kernel.
Let
\[
C_0(\T) = \{ g \in C(\T) \mid g(-1)=g(1)=0 \}.
\]
\begin{prop} \label{prop:reproducing property}
For $g \in C_0(\T)$
\[
\lim_{l \to \infty} \frac{1}{4\pi i} \int_\T g(\la) \langle \phi_\la,\phi_{\la'} \rangle_l \frac{d\la}{\la} =
\begin{cases}
K\, g(\la') |c(\la')|^2,& \la'\in \T\setminus\{-1,1\},\\
0, & \la' \in S.
\end{cases}
\]
\end{prop}
\begin{proof}
First consider $\la'= e^{i\te'}$ with $\te'\in(0,\pi)$. Using Lemma \ref{lem:properties truncated inner product} and substitution we have, for $l \to \infty$,
\[
\begin{split}
I_l[g](\la') &= \frac{1}{4\pi i} \int_\T g(\la) \langle \phi_\la,\phi_{\la'} \rangle_l \frac{d\la}{\la}\\
&= \frac{K}{2\pi} \sum_{\eps,\eta \in \{-1,1\}} \int_0^\pi g(e^{i\te}) \Big(F_l^{\eps,\eta}(\te,\te')+\mathcal O(q^l)\Big) \, d\te,
\end{split}
\]
where
\[
F_l^{\eps,\eta}(\te,\te') =\frac{(e^{i\eps\te} - e^{i \eta \te'})e^{il(\eps+\eta)}  c(e^{i\eps \te}) c(e^{i\eta \te'})}{2\cos(\te)-2\cos(\te')}.
\]
Since $\la \mapsto c(\la)$ is continuous on $\T\setminus\{-1,1\}$, the function $\te \mapsto F_l^{\eps,\eta}(\te,\te')$ is continuous on $(0,\pi) \setminus\{\te'\}$. If $\sgn(\eps) = \sgn(\eta)$ the singularity at $\te'$ is removable, so that these terms vanish in the limit by the Riemann-Lebesgue lemma. This gives
\[
\lim_{l \to \infty} I_l[g](e^{i\te'}) = \lim_{l \to \infty} \frac{K}{2\pi} \int_0^\pi g(e^{i\te}) \Big(F_l^{1,-1}(\te,\te')+F_l^{-1,1}(\te,\te')\Big) \, d\te,
\]
where dominated convergence is used to get rid of the $\mathcal O(q^l)$-terms. Furthermore, using trigonometric identities we obtain
\begin{equation} \label{eq:F+F}
\begin{split}
F_l^{1,-1}(\te,\te')+F_l^{-1,1}(\te,\te') &= \frac{[c(e^{-i\te})c(e^{i\te'}) - c(e^{i\te})c(e^{-i\te'})](e^{-i\te}-e^{i\te'})}{4 \sin(\hf(\te+\te')) \sin(\hf(\te-\te'))}  e^{il (\te'-\te)} \\
& \quad + c(e^{i\te})c(e^{i\te'}) D_l(\te-\te'),
\end{split}
\end{equation}
where
\[
D_l(t)= \frac{ \sin((2l+1)t) }{ \sin(\hf t)}
\]
is the Dirichlet kernel. Note that the first term in \eqref{eq:F+F} has a removable singularity at $\te = \te'$, so using the Riemann-Lebesgue lemma again this term vanishes in the limit $l \to \infty$. Then from the well-known limit property of the Dirichlet kernel we obtain
\[
\lim_{l \to \infty} I_l[g](e^{i\te'}) = K g(e^{i\te'}) c(e^{i\te'}) c(e^{-i\te'}),
\]
which gives the result for $\la' \in \T\setminus\{-1,1\}$.

Next let $\la' \in S$. In this case $c(1/\la')=0$ and then since $|\la'|<1$, we obtain from Lemma \ref{lem:properties truncated inner product}
\[
\lim_{l \to \infty}\langle \phi_\la,\phi_{\la'} \rangle_l = K\lim_{l \to \infty}\sum_{\eps \in \{-1,1\}} \frac{(\la^{\eps} - \la')(\la^{\eps} \la' )^l c(\la^{\eps}) c(\la')}{\mu(\la)-\mu(\la')}\Big(1+\mathcal O(q^l)\Big)=0,
\]
from which the result follows.
\end{proof}

Now we are ready to prove Theorem \ref{thm:main}.
\begin{proof}[Proof of Theorem \ref{thm:main}].
From Proposition \ref{prop:isometry} and Lemma \ref{lem:GF=id} we already know that $\mathcal F$ is an isometry with left-inverse $\mathcal G$, so
we have to show that $\mathcal G$ is a right-inverse of $\mathcal F$.
Let $\mathcal H_0 \subseteq \mathcal H$ be the dense subspace consisting of functions $g \in \mathcal H$ such that $g|_\T \in C_0(\T)$ and $g|_S$ has finite support. Then for $\la'\in \T \cup S$,
\[
\begin{split}
(\mathcal F \mathcal G g)(\la') &= \int_{-1}^{\infty(1)} \phi_{\la'}(x)\int g(\la) \phi_\la(x) \, d\nu(\la) \,w(x)\, d_qx\\
& = \lim_{l \to \infty} \frac{1}{4K \pi i} \int_\T g(\la) \langle \phi_{\la},\phi_{\la'} \rangle_l \frac{d\la}{\la|c(\la)|^2} \\
& \quad + \frac{1}{K}\sum_{\la \in S} g(\la) \langle \phi_{\la}, \phi_{\la'} \rangle \Res{\hat \la = \la}\left( \frac{1}{\hat \la c(\hat \la) c(\hat\la^{-1})} \right).
\end{split}
\]
Using Proposition \ref{prop:reproducing property} for the integral part and Corollary \ref{cor:orthogonality relations} for the sum part, we find that this equals $g(\la')$. Since $\mathcal H_0$ is dense in $\mathcal H$ we find $\mathcal F \mathcal G = \mathrm{Id}_{\mathcal H}$.
\end{proof}

\end{document}